\documentclass[reqno]{amsart}
\usepackage{amsmath}
\usepackage{amssymb}
\usepackage{arydshln}
\usepackage{tensor}
\usepackage{appendix}%
\usepackage{bookmark}

\allowdisplaybreaks[3]
\numberwithin{equation}{section}

\newtheorem{thm}{\indent Theorem}[section]
\newtheorem{cor}[thm]{\indent Corollary}

\newtheorem{prop}[thm]{\indent Proposition}
\newtheorem{dfn}{{\indent\bf Definition}}[section]
\newtheorem{rmk}{{\indent\bf Remark}}[section]
\newtheorem{expl}{{\indent\bf Example}}[section]

\newcommand{\hs}{\hspace}

\newcommand{\td}{\tilde}

\newcommand{\fr}{\frac}

\newcommand{\edd}{\end{document}}
\newcommand{\be}{\begin{equation}}
\newcommand{\ee}{\end{equation}}

\newcommand{\lmx}{\left(\begin{matrix}}
\newcommand{\rmx}{\end{matrix}\right)}
\newcommand{\ldt}{\left|\begin{matrix}}
\newcommand{\rdt}{\end{matrix}\right|}

\newcommand{\hess}{{\rm Hess\,}}

\newcommand{\tr}{{\rm tr\,}}
\newcommand{\vol}{{\rm Vol}}
\newcommand{\gl}{{\rm GL}}
\newcommand{\iso}{{\rm Iso}}

\newcommand{\const}{{\rm const}}

\newcommand{\bbr}{{\mathbb R}}

\newcommand{\mo}{M\"obius }
\newcommand{\ba}{\begin{array}}
\newcommand{\ea}{\end{array}}
\newcommand{\nnm}{\nonumber}
\newcommand{\beal}{\begin{align}}
\newcommand{\eal}{\end{align}}
\newcommand{\bea}{\begin{eqnarray}}
\newcommand{\eea}{\end{eqnarray}}

\newcommand{\pp}[2]{\fr{\partial #1}{\partial #2}}
\newcommand{\ppp}[3]{\fr{\partial^2 #1}{\partial #2\partial #3}}

\newcommand{\tensorscript}[2]{\tensor[#1]{#2}{}\hs{-.3mm}}
\begin{document}

\title[Equiaffine isoparametric functions and their level hypersurfaces]{Equiaffine isoparametric functions and their regular level hypersurfaces}

\author[X. X. Li]{Xingxiao Li$^*$}

\author[W. J. Hao]{Wenjing Hao}

\dedicatory{}

\thanks{Research supported by
National Natural Science Foundation of China (No. 11671121 and No. 11371018).}
\address{
School of Mathematics and Information Sciences
\endgraf Henan Normal University \endgraf Xinxiang 453007, Henan, P.R. China
}
\email{xxl$@$henannu.edu.cn; haojw0223@163.com}


\begin{abstract}
In this paper, we introduce and study the locally strongly convex equiaffine isoparametric hypersurfaces and equiaffine isoparametric functions on the affine space $A^{n+1}$. Motivated by the case on the Euclidean space $E^{n+1}$, we first introduce the concept of equiaffine parallel hypersurfaces in $A^{n+1}$, obtaining some fundamental identities with the basic equiaffine geometric invariants, and then we define the equiaffine isoparametric hypersurfaces to be ones that are among families of equiaffine parallel hypersurfaces of constant affine mean curvature in $A^{n+1}$. Finally, we introduce the concept of equiaffine isoparametric functions on $A^{n+1}$, and prove that any equiaffine isoparametric hypersurface is exactly a regular level set of some equiaffine isoparametric function.
\end{abstract}

\maketitle


\section{Introduction}\label{s1}
For a long period of time, the study of isoparametric functions (isoparametric hypersurfaces and their focal submanifolds) has been a highly influential field in differential geometry. In the history of this subject, E. Cartan was the pioneer who made a comprehensive study of isoparametric functions (hypersurfaces) on the real space forms. Originally, a hypersurface on real space forms was said to be isoparametric if it is of constant principal curvature. Due to Cartan \cite{cart38}, \cite{cart39}) and M\"unzner (\cite{munz80}, \cite{munz81}), these isoparametric hypersurfaces, especially, of standard Euclidean spheres, became fascinating to study and are closely related to a class of smooth functions satisfying certain equations called Cartan-M\"unzner equations, which we call isoparametric functions (or Cartan polynomial). In fact, any isoparametric hypersurface in real space forms must be a regular level set of an isoparametric function and, conversely, any level set of an isoparametric function is among a family of parallel isoparametric hypersurfaces (cf. \cite{ce-ry}). Later, more generally in a Riemannian manifold $(N,\td g)$, an isoparametric hypersurface was naturally defined as the regular level set of an isoparametric function, equivalent to that it is among a (smooth) family of parallel hypersurfaces of constant mean curvature (\cite{wqm87}). In case that $(N,\td g)$ is not of constant sectional curvature, the conclusion that an isoparametric hypersurface is of constant principal curvature is not true any more in general. In fact, as is known, an isoparametric hypersurface in the complex projective space can not be of constant principal curvature (\cite{wqm80}). Note that there are several nice systematic surveys on isoparametric functions, isoparametric hypersurfaces and their generalizations (\cite{thor00}, \cite{ce08} and \cite{ch12}). See also \cite{c-c-j}, \cite{g-x}, \cite{im}, \cite{g-t13} and \cite{g-t14} for recent progresses and applications.

On the other hand, there are also some other geometric theories of submanifolds in which the concept of isoparametric hypersurfaces appears, being introduced by the constancy of some certain ``principal curvatures''. For example, in the M\"obius geometry of submanifolds in spheres, the M\"obius and Blaschke isoparametric hypersurfaces are systematically studied, and some interesting classification theorems are obtained (see \cite{hlz}, \cite{llwz}, \cite{ltqw}, \cite{ltw}, \cite{lz07}, \cite{lz09}, \cite{rt}, and so on). Then, in a rather broad field of geometry, it is much interesting to find certain relevant isoparametric functions of which the regular level sets are exactly those isoparametric hypersurfaces in geometries other than the Riemannian geometry, say, for both M\"obius and Blaschke isoparametric hypersurfaces.

So, it is also very natural for us to consider affine isoparametric hypersurfaces in the affine geometry of hypersurfaces. This class of hypersurfaces should be assumed to contain some ``nice'' examples, say, homogeneous hypersurfaces and, in particular, the affine hyperspheres which, as we know, are the most important objects in this area.

Motivated by isoparametric hypersurfaces in the Euclidean space, we aim in this paper to introduce, in a natural and reasonable manner, the concept of (equi-)affine isoparametric hypersurfaces and that of (equi-)affine isoparametric functions, initiating a study on them, especially giving a close relation between them. For simplicity, we shall now only consider the locally strongly convex hypersurfaces.

The first things we shall do is introducing the concept of equiaffine parallel hypersurfaces in the affine space $A^{n+1}$ (Definition \ref{equiprlhyps}), proving other equivalent definitions (Proposition \ref{prop3.1} and Corollary \ref{cor3.2}) and obtaining some fundamental identities for the basic equiaffine geometric invariants of these parallel hypersurfaces. After this, we define (in Definition \ref{isop hyps}) an equiaffine isoparametric hypersurface to be the one that is among a family of equiaffine parallel hypersurfaces of constant affine mean curvature in $A^{n+1}$. These are the content of Section \ref{s3}.

The major part of this paper is in Section \ref{s4} where we shall first introduce the concept of equiaffine isoparametric functions on $A^{n+1}$ (Definition \ref{isop hyps}), and then prove the following main theorems:

\begin{thm}\label{main1}
A non-degenerate hypersurface $x:M^n\to A^{n+1}$ in the affine space $A^{n+1}$ is (locally strongly convex) equiaffine isoparametric if and only if it is a regular level set of an equiaffine isoparametric function defined on an open neighbourhood of $x(M^n)$ in $A^{n+1}$.
\end{thm}

While the definition (Definition \ref{isop hyps}) of the equiaffine isoparametric hypersurface is a geometric one, the above main theorem gives an analytic description for the same concept.

\begin{thm}\label{main2}
A non-degenerate hypersurface in the affine space $A^{n+1}$ is equiaffine isoparametric if and only if it is of constant affine principal curvature.
\end{thm}

Theorem \ref{main2} is exactly what we really expect to have. We also have the following corollaries:

\begin{cor}[Corollaries \ref{cor4.1} and \ref{cor4.1'}]
A non-degenerate hypersurface in the affine space $A^{n+1}$ is equiaffine isoparametric if and only if it is among a family of equiaffine parallel hypersurfaces.
\end{cor}

\begin{cor}[Corollary \ref{cor4.2}]
A locally strongly convex hypersurface in the affine space $A^{n+1}$ is an (equi)affine hypersphere if and only if it is among a family of equiaffine parallel affine hyperspheres of the same type.
\end{cor}

\begin{rmk}\rm The very last corollary may provide a new direction of insight to deal with the classification problem for affine hyperspheres.
\end{rmk}

The first author thanks Professor A.-M. Li for his constant encouragement.

\section{Equiaffine differential geometry of hypersurfaces}\label{s2}
In this section, we brief some basic facts for the equiaffine differential geometry of locally strongly convex hypersurfaces, including the necessary notations we shall use in this paper. For more details of this, we refer the readers to \cite{luz} and \cite{ns}.

Let $\bbr$ be the field of real numbers and $\bbr^{n+1}$ be the real vector space of all ordered $(n+1)$-tuples of real numbers, that is,
$$
\bbr^{n+1}=\{x=(x^1,\cdots,x^n,x^{n+1});\ x^1,\cdots,x^n,x^{n+1}\in\bbr\}.
$$
Then on $\bbr^{n+1}$ there are a canonical flat connection $d$ defined by the usual component-wise differentiation of $\bbr^{n+1}$-valued functions, and a canonical volume measure $\vol$ defined by the determinant function of $n+1$ vectors. Endowed with the connection $d$ and the volume measure $\vol$, $\bbr^{n+1}$ is taken to be a measured affine space which we denote by $A^{n+1}$.

Note that the group $\gl(n+1)$ of linear transformations on $\bbr^{n+1}$ together with the additive group $\bbr^{n+1}$ of translations on $\bbr^{n+1}$ makes into a semi-direct product group $A(n+1):=\gl(n+1)\ltimes \bbr^{n+1}$, called the {\em affine transformation} group on $A^{n+1}$. An element $T\in A(n+1)$ is called a {\em uni-modular transformation} if it preserves the volume measure $\vol$ of $A^{n+1}$. We denote by ${\rm UA}(n+1)$ the subgroup of $A(n+1)$ consisting of all the uni-modular transformations on $A^{n+1}$. Moreover, given an open domain $U\subset A^{n+1}$, we call a Riemannian metric $\td g$ on $U$ {\em admissible} to the volume measure $\vol$, if the isometry group $\iso(\td g)\subset {\rm UA}(n+1)$.

Since the tangent bundle $TA^{n+1}=A^{n+1}\times\bbr^{n+1}$, for a tangent frame field $\{e_1,\cdots,e_{n+1}\}$ on $A^{n+1}$, the volume function $\vol(e_1,\cdots,e_{n+1})$ of $\{e_1,\cdots,e_{n+1}\}$ makes sense. We call $\{e_1,\cdots,e_{n+1}\}$ {\em uni-modular} if $\vol(e_1,\cdots,e_{n+1})\equiv 1$.

For simplicity, the following ranges of indices are always assumed without further specification in the present paper:
$$
1\leq i,j,k,\cdots\leq n;\quad 1\leq A,B,C,\cdots\leq n+1.
$$

Now let $x:M^n\to A^{n+1}$ be an $n$-dimensional immersion of manifold $M^n$ into $A^{n+1}$. Then we simply call $x$ or $x(M^n)$ a hypersurface of $A^{n+1}$. A frame field $\{e_1,\cdots,e_{n+1}\}$ around $x(M^n)$ is called {\em uni-modular affine Darbourx}, or simply, {\em affine Darbourx}, if it is uni-modular and, when restricted to $x(M^n)$, $e_1,\cdots,e_n$ are tangent to $x(M^n)$. In this case, $\{e_1,\cdots,e_n\}$, by restriction to $x(M^n)$ and then via the tangent map $x_*$, locally defines a frame field on $M^n$ which we still denote as $\{e_1,\cdots,e_n\}$. By fixing a (uni-modular) affine Darbourx frame field $\{e_1,\cdots,e_{n+1}\}$, one has the local decomposition of vector bundle:
$$
x^*TA^{n+1}=x_*(TM^n)\oplus\bbr\cdot e_{n+1}.
$$
So that
\be\label{gauss-e}
x_{ij}\equiv e_j(e_i(x))=\sum\Gamma^k_{ij}x_k+h_{ij}e_{n+1},\quad\forall\,i,j
\ee
where $\Gamma^k_{ij}$ and $h_{ij}$ are local smooth functions. The (induced) connection on $M^n$ given by the coefficients $\Gamma^k_{ij}$ is called the affine connection of $x$, while \eqref{gauss-e} is called the {\em affine Gauss formula}.

When $x$ is {\em non-degenerate}, that is, the matrix $(h_{ij})$ is non-singular everywhere on $M^n$, the locally defined function $H=|\det(h_{ij})|>0$. Define $G_{ij}=H^{-\fr1{n+2}}h_{ij}$ for each pair of $i,j$. Then $G=\sum G_{ij}\omega^i\omega^j$ is a well-defined Pseudo-Riemannian metric on $M^n$ (\cite{luz}), where $\omega^1,\cdots,\omega^n,\omega^{n+1}$ is the dual frame field of $\{e_1,\cdots,e_{n+1}\}$. In particular, if $x$ is {\em locally strongly convex}, or the matrix $(h_{ij})$ is definite everywhere, we can suitably choose the orientation to make the matrix $(h_{ij})$ and hence the metric $G$ positive definite. Conventionally the metric $G$ defined in this way is called the {\em Blaschke metric} or {\em affine metric}, by which the affine normal vector $Y$ of $x$ is defined as
\be
Y=\fr1n\Delta_Gx.
\ee
For a given point $u\in M$, the straight line passing through $x(u)$ and parallel to $Y$ is called the affine normal line of $x$ at $u$.

Let $\td\Gamma^k_{ij}$ be the coefficients of the Levi-Civita connection of $G$ with respect to $\{e_1,\cdots,e_n\}$. Then the $(1,2)$-tensor $A$ defined by $A^k_{ij}=\Gamma^k_{ij}-\td\Gamma^k_{ij}$ is called the difference tensor. The difference tensor is identified, via the metric $G$, with a symmetric $3$-form
$$A=\sum A_{ijk}\omega^i\omega^j\omega^k,\quad A_{ijk}=\sum G_{kl}A^l_{ij},$$
which is called the {\em Fubini-Pick form}, or {\em the cubic form}, or the {\em affine third fundamental form}. Furthermore, if the one-forms $\omega^1_{n+1},\cdots,\omega^{n+1}_{n+1}$ are given by $$de_{n+1}=\sum\omega^i_{n+1}x_i+\omega^{n+1}_{n+1}e_{n+1},$$ Then we have (\cite{luz})

\begin{prop}\label{prop2.1} Let $\{e_1,\cdots,e_{n+1}\}$ be an arbitrary affine Darbourx frame field of the hypersurface $x$. Then the following three conditions are equivalent to each other:

(1) $e_{n+1}$ is parallel to the affine normal vector $Y$;

(2) $\tr_GA\equiv\sum G^{ij}A^k_{ij}=0$ where $(G^{ij})=(G_{ij})^{-1}$;

(3) $\omega^{n+1}_{n+1}+\fr1{n+2}d\log H=0$.
\end{prop}

Moreover, when $e_{n+1}$ is parallel to $Y$ we have \be\label{y-y}
Y=H^{\fr1{n+2}}e_{n+1},\quad\vol(x_1,\cdots,x_n,Y)=H^{\fr1{n+2}}.
\ee
Thus, by using the affine normal vector $Y$, the affine Gauss formula \eqref{gauss-e} can be written as
\be\label{gauss-y}
x_{ij}=\sum\Gamma^k_{ij}x_k+G_{ij}Y,\quad\forall\,i,j
\ee

On the other hand, the affine Weingarten map or the affine shape operator $B=(B^j_i)$ is defined by
\be\label{yi}
Y_i=-\sum B^j_ix_j,\quad\forall\,i.
\ee
Accordingly, the corresponding $2$-form $B=\sum B_{ij}\omega^i\omega^j$ with $B_{ij}=\sum G_{ik}B^k_j$ is called the {\em affine (second) fundamental form}. Furthermore, the eigenvalues $\lambda_1,\cdots,\lambda_n$ of the matrix $(B^j_i)$, which are globally well-defined on $M^n$, are called the {\em affine principal curvatures} of $x$, from which a class of important and interesting hypersurfaces in the affine geometry are defined as follows.

\begin{dfn}
A locally strongly convex hypersurface is called an (equi)affine hypersphere if all of its affine principal curvatures are equal to a same constant $\lambda$. Furthermore, an affine hypersphere is called elliptic (resp. parabolic, hyperbolic) if $\lambda>0$ (resp. $\lambda=0$, $\lambda<0$).
\end{dfn}

\begin{dfn}\label{affmeancur}
The affine mean curvature of the hypersurface $x$, denoted by $L_1$ is defined as
\be\label{afmeancur}
L_1=\fr1n\tr_GB\equiv\fr1n\sum B^i_i=\fr1n\sum G^{ij}B_{ij} =\fr1n\sum_i\lambda_i.
\ee
Moreover, $x$ is called affine maximal if $L_1\equiv 0$.
\end{dfn}

Finally, the affine Gauss equation, that relates the Riemannian curvature tensor $R_{ijkl}$ to the fundamental affine invariants $G$, $A$ and $B$,  and the affine Codazzi equations are expressed as follows (\cite{luz}):
\begin{align}
R_{ijkl}=&\sum(A^m_{ik}A_{mjl}-A^m_{il}A_{mjk})\\
&+\fr12(G_{il}B_{jk}+G_{jk}B_{il}-G_{ik}B_{jl}-G_{jl}B_{ik},\\
A_{ijk,l}-&A_{ijl,k}=\fr12(G_{ik}B_{jl}+G_{jk}B_{il} -G_{il}B_{jk}-G_{jl}B_{ik}),\\
B_{ij,k}-&B_{ik,j}=\sum(A^l_{ij}B_{kl}-A^l_{ik}B_{jl})
\end{align}
for all $i,j,k,l$.

\section{Equiaffine parallel hypersurfaces and equiaffine isoparametric hypersurfaces}\label{s3}
\newcommand{\tlmu}[1]{\tensorscript{^\mu}{#1}}
\newcommand{\tlmux}{\tlmu{x}}
\newcommand{\tlmuy}{\tlmu{Y}}
\newcommand{\tlmug}{\tlmu{G}}
\newcommand{\tlmua}{\tlmu{A}}
\newcommand{\tlmub}{\tlmu{B}}
\newcommand{\tlmul}{\tlmu{L}}
\newcommand{\tlmugm}{\tlmu{\Gamma}}
\newcommand{\tlmuh}{\tlmu{H}}
\newcommand{\tlmut}{\tlmu{T}}
\newcommand{\tlmutdgm}{\tlmu{\td\Gamma}}
\newcommand{\tlmulmbd}{\tlmu{\lambda}}
Let $x:M^n\to A^{n+1}$ be a locally strongly convex hypersurface with the affine metric $G$, the Fubini-Pick form $A$, the affine second fundamental form $B$ and the affine normal vector $Y$. For each given $\mu\in C^\infty(M^n)$, we can define a new hypersurface $\tlmux$ as follows:

\be\label{paral hyps}
\tlmux(u)=x(u)+\mu(u)Y(u),\quad u\in M^n.
\ee

In what follows in this paper, we always use a left superscript $\mu$ to an affine invariant $\Gamma$, that is, $\tlmu{\Gamma}$, to denote the corresponding invariant for the hypersurface $\tlmux$. For example, the affine metric, the affine normal vector, the difference tensor or the Fubini-Pick form, the affine Weingarten map or the affine fundamental form and the affine mean curvature are denoted by $\tlmug$, $\tlmuy$, $\tlmua$, $\tlmub$ and $\tlmul_1$, respectively.

\begin{dfn}\label{prlhyps}
A hypersurface $\tlmux:M^n\to A^{n+1}$ defined by \eqref{paral hyps} is called affine parallel to $x:M^n\to A^{n+1}$ if

(1) $\tlmux$ is locally strongly convex;

(2) the tangent spaces $\tlmux_{*u}(T_uM^n)$ and $x_{*u}(T_uM^n)$ are parallel in $A^{n+1}$ for all $u\in M^n$.
\end{dfn}

Then the following corollary is direct from the above definition:
\begin{cor}\label{cor3.1}
Let a hypersurface $\tlmux$ be given by \eqref{paral hyps}. Then $\tlmux$ is affine parallel to $x$ if and only if the function $\mu$ is a constant one.
\end{cor}

Indeed, from \eqref{paral hyps}, it is direct that
\be\label{olxi}
\tlmux_i=x_i+\mu_iY+\mu Y_i=\sum (\delta^j_i-\mu B^j_i)x_j+\mu_iY,\quad\forall\,i.
\ee
Thus $\tlmux$ is affine parallel to $x$ if and only if
$\mu_i=0$ for each $i$, that is, $\mu$
is a constant.

Now we are concerned with the equiaffine geometry, so the following definition is more relevant here:

\begin{dfn}\label{equiprlhyps}
A hypersurface $\tlmux:M^n\to A^{n+1}$ defined by \eqref{paral hyps} is called equiaffine parallel to $x:M^n\to A^{n+1}$ if

(1) $\tlmux$ is affine parallel to $x$;

(2) the affine normal line of $\tlmux$ is parallel to that of $x$ everywhere on $M^n$.
\end{dfn}

\begin{rmk}\rm
(i) It seems natural to seek conditions under which either of the conditions (1) and (2) in Definition \ref{equiprlhyps} will imply the other.

(ii) For relatively (affine) parallel hypersurfaces, we refer the readers to \cite{s-k-d}, \cite{s-k2}, \cite{s-k3} where, by means of the standard Euclidean metric on the ambient space and support functions, the authors obtain many facts on the relatively parallel hypersurfaces, including some curvature conditions for relatively parallel hypersurfaces to have parallel (equi-)affine normal lines.
\end{rmk}

In what follows, we denote by $\tlmut\equiv(\tlmut^i_j)$ the matrix with $\tlmut^i_j=\delta^i_j-\mu B^i_j$. Then, by \eqref{olxi}, $\tlmux$ is affine parallel to $x$ is and only if
\be\label{muxi}
\tlmux_i=\sum\tlmut^j_ix_j,\quad \forall\,i.
\ee

\begin{prop}\label{prop3.1}
Let the hypersurface $\tlmux:M^n\to A^{n+1}$ given in \eqref{paral hyps} be affine parallel to $x$. Then $\tlmux$ is equiaffine parallel to $x$ if and only if there is a constant $c\in \bbr$ depending on $\mu$ such that, the affine normal vector $\tlmuy(u)=cY(u)$ for every $u\in M^n$.
\end{prop}

\begin{proof}
The assumption that $\tlmux$ is affine parallel to $x$ is, by Corollary \ref{cor3.1}, equivalent to that $\mu$ is constant.

Suppose that $\tlmuy=\sum a^ix_i+cY$ for some functions $a^i, c$ on $M^n$. Then, by \eqref{muxi} and the affine Gauss formula, we directly compute
\begin{align}
\tlmux_{ij}=&\left(\sum\tlmut^k_ix_k\right)_j=\sum (\tlmut^k_i)_jx_k+\sum\tlmut^k_ix_{kj}\nnm\\
=&\sum(\tlmut^k_i)_jx_k +\sum\tlmut^l_i\left(\sum\Gamma^k_{lj}x_k+G_{lj}Y\right)\nnm\\ =&\sum\left((\tlmut^k_i)_j +\sum\tlmut^l_i\Gamma^k_{lj}\right)x_k +\sum\tlmut^k_iG_{kj}Y.
\end{align}
On the other hand,
\begin{align}
\tlmux_{ij}=&\sum\tlmugm^k_{ij}\tlmux_k+\tlmug_{ij}\tlmuy\nnm\\
=&\sum\tlmugm^l_{ij}\left(\sum \tlmut^k_lx_k\right) +\tlmug_{ij}\left(\sum a^kx_k+cY\right)\nnm\\
=&\sum\left(a^k\tlmug_{ij}+\sum\tlmugm^l_{ij} \tlmut^k_l\right)x_k +c\,\tlmug_{ij}Y.
\end{align}
Comparing the above two equalities we obtain for any $i,j$
\be\label{mutg=cmug}
(\tlmut^k_i)_j +\sum\tlmut^l_i\Gamma^k_{lj} =a^k\tlmug_{ij}+\sum\tlmugm^l_{ij}\tlmut^k_l,\quad \forall\,k,\quad \sum\tlmut^k_iG_{kj}=c\,\tlmug_{ij}.
\ee

Moreover, we have
\begin{align}
\tlmuy_i=&\left(\sum a^jx_j+cY\right)_i=\sum (e_i(a^j)x_j +a^jx_{ji})+c_iY+cY_i\nnm\\
=&\sum \left(e_i(a^j)+\sum a^k\Gamma^j_{ki}-cB^j_i\right)x_j +\left(\sum a^jG_{ji}+c_i\right)Y,
\end{align}
and
\be
\tlmuy_i=-\sum\tlmub^j_i\tlmux_j =-\sum\tlmub^k_i\tlmut^j_kx_j.
\ee
So it holds that
\be\label{bi}
e_i(a^j)+\sum a^k\Gamma^j_{ki}-cB^j_i+\sum\tlmub^k_i\tlmut^j_k=0,\quad \sum a^jG_{ji}+c_i=0.
\ee

Furthermore, by the second formula of \eqref{y-y}, it holds that
\begin{align}
\tlmuh^{\fr1{n+2}} =&\vol(\tlmux_1,\cdots,\tlmux_n,\tlmuy)\nnm\\ =&\vol\left(\sum\tlmut^i_1x_i,\cdots,\sum\tlmut^i_nx_i,\sum a^ix_i+cY\right)\nnm\\
=&\det(\tlmut)cH^{\fr1{n+2}}.\label{muhandh}
\end{align}

Since $G_{ij}=H^{-\fr1{n+2}}h_{ij}$ and $\tlmug_{ij}=\tlmuh^{\fr1{n+2}}\tlmu{h}_{ij}$, implying
\be\label{detg}
\det G=H^{\fr2{n+2}},\quad \det\tlmug=\tlmuh^{\fr2{n+2}},
\ee
it is easily seen from the second equality of \eqref{mutg=cmug} that $c^n\tlmuh^{\fr2{n+2}}=\det(\tlmut)H^{\fr2{n+2}}$. This last equality together with \eqref{muhandh} gives
\be\label{c}
c^{n+2}\det(\tlmut)=1,\quad c^{n+1}=\left(\fr{H}{\tlmuh}\right)^{\fr1{n+2}},\quad
\det(\tlmut)=\left(\fr{H}{\tlmuh}\right)^{-\fr1{n+1}}.
\ee

Apparently, we only need to prove the necessity part of the proposition. Suppose that $\tlmux$ is equiaffine to $x$. Then $\tlmuy$ is parallel to $Y$ which is equivalent to that $a^i=0$, $\forall\,i$. This with the second equality of \eqref{bi} shows that $\tlmuy=cY$ with $c$ being a constant.
\end{proof}

Denote by $L_r$ the $r$-th normalized elementary symmetric functions of the affine principal curvatures $\lambda_1,\cdots,\lambda_n$:
$$
L_r=\fr1{C^r_n}\sum_{1\leq i_1<\cdots<i_r\leq n} \lambda_{i_1}\cdots\lambda_{i_r},\quad 1\leq r\leq n.
$$
Then the following corollary is direct from Proposition \ref{prop3.1}, \eqref{detg} and \eqref{c}:

\begin{cor}\label{cor3.2}
For an hypersurface $\tlmux$ that is affine parallel to $x$, the following four conditions are equivalent to each other:

(1) $\tlmux$ is equiaffine parallel to $x$;

(2) $\det(\tlmut)$ is constant, which is equivalent to that
\be\label{dett=const}
nL_1-C^2_n\mu L_2+\cdots+(-1)^nC^{n-1}_n\mu^{n-2}L_{n-1}+(-1)^{n+1} \mu^{n-1}K=const;
\ee

(3) The ratio $\fr{H}{\tlmuh}$ of the functions $H$ and $\tlmuh$ is constant;

(4) The ratio $\fr{\det G}{\det\tlmug}$ of the squared volumes $\det G$ and $\det\tlmug$ is constant.
\end{cor}

From Proposition \ref{prop3.1}, we clearly have

\begin{cor}\label{cor3.3}
$\tlmux$ is equiaffine parallel to $x$ if and only if $x$ is equiaffine parallel to $\tlmux$.

\em{If it is the case, we shall say that $x$ and $\tlmux$ are equiaffine parallel to each other.}
\end{cor}

\begin{prop}\label{prop3.5}
If a locally strongly convex hypersurfaces is of constant affine principal curvature, then all of its equiaffine parallel hypersurfaces are also of constant affine principal curvature.
\end{prop}

\begin{proof}
Since $\tlmuy$ is parallel to $Y$ for all $\mu$, the second equality of $\eqref{bi}$ becomes \be\label{mub-cb}\sum\tlmub^k_i\tlmut^j_k=cB^j_i.\ee
This with the second equality of \eqref{mutg=cmug} gives
$$
c\,\tlmub_{ij}=c\sum\tlmub^k_i\tlmug_{kj} =\sum\tlmub^k_i\tlmut^l_kG_{lj} =c\sum B^k_iG_{kj}=cB_{ij}.
$$
So \eqref{mub-cb} is equivalent to that
\be\label{mub=b}\tlmub_{ij}=B_{ij}.\ee

Now choose a suitable frame field $\{e_1,\cdots,e_n\}$ such that $B^j_i=\lambda_i\delta^j_i$. Then $$\tlmut^j_k=(1-\mu\lambda_k)\delta^j_k.$$
Putting this into \eqref{mub-cb} we obtain
$$
\tlmub^j_i(1-\mu\lambda_j) =\sum\tlmub^k_i(1-\mu\lambda_k)\delta^j_k =c\lambda_i\delta^j_i,
$$
implying
$$
\tlmub^j_i=\fr{c\lambda_i}{1-\mu\lambda_i}\delta^j_i.
$$
Thus $e_1,\cdots,e_n$ are also the eigen-vectors of the affine Weingarten map $\tlmub$ of $\tlmux$, and the corresponding affine principal curvatures of $\tlmux$ are, respectively, \be\label{mulmbd}
\tlmulmbd_i=\fr{c\lambda_i}{1-\mu\lambda_i},\quad \forall\,i.
\ee
The conclusion of the proposition follows easily, since $c$ is constant along $\tlmux$.
\end{proof}

Motivated by the isoparametric hypersurfaces in Riemannian manifolds, we naturally introduce the equiaffine isoparametric hypersurfaces as follows:

\begin{dfn}\label{isop hyps}
A locally strongly convex hypersurface $x:M\to A^{n+1}$ is called equiaffine isoparametric if it is among a family of equiaffine parallel hypersurfaces that are of constant affine mean curvature.
\end{dfn}

\begin{expl}\label{expl1 s3-2}\rm
All equiaffine hyperspheres are equiaffine isoparametric hypersurfaces. In fact, the reason of this is rather simple (see Theorem \ref{main2}).
\end{expl}

More broadly, due to the same reason, we have

\begin{expl}\label{expl2 s3-2}\rm
All equiaffine homogeneous hypersurfaces are equiaffine isoparametric hypersurfaces.
\end{expl}

\begin{rmk}\rm So, two questions may be asked: Firstly, are there any equiaffine isoparametric hypersurfaces in $A^{n+1}$ other than the affine hyperspheres or, more broadly, other than the equiaffine homogeneous hypersurfaces? Secondly, if either of the answers is positive, then the corresponding classification problem arises naturally which seems interesting!
\end{rmk}

\section{Equiaffine isoparametric functions and their regular level sets}\label{s4}
First of all, we recall the definition of isoparametric functions on a Riemannian manifold.

Let $(N,\td g)$ be a Riemannian manifold, and $F$ be a non-constant smooth function on $N$ with the gradient $\td\nabla^{\td g} F$ and the Laplacian $\td\Delta_{\td g} F$. Denote $J=F(N)\subset\bbr$. Then, according to Q. M. Wang (\cite{wqm87}), $F$ is called an isoparametric function if there exist smooth functions $a(t)$ and $b(t)$ of one variable $t$, $t\in J$, such that
\be
|\td\nabla^{\td g} F|_{\td g}=a(F),\quad \td\Delta_{\td g}=b(F).
\ee

Define
\be\label{y}
N_*=\{x\in N;\ dF\neq 0\},\quad \xi=\fr1a\td\nabla^{\td g} F.
\ee
Then $\xi$ is a globally defined smooth unit vector field on $N_*$, of which the orthogonal complement $\xi^\bot$ is a distribution on $N_*$. Furthermore, we have the following well-known conclusion (\cite{cart39}, \cite{wqm87})

\begin{prop}\label{prop4.1}
Let $F$ be an isoparametric function on $(N,\td g)$. Then

(1) Each integral curves $\gamma_\xi$ of $\xi$ is a unit speed geodesic with the arc-length being its parameter;

(2) The distribution $\xi^\bot$ is integrable, generating a foliation $\mathcal F$ of $N_*$ by a family of hypersurfaces of constant mean curvature;

(3) Each geodesic $\gamma_\xi$ intersects any hypersurface in $\mathcal F$ (orthogonally) once and only once;

(4) Every two hypersurfaces in $\mathcal F$ are of equi-distance, that is, all the geodesics $\gamma_\xi$ restricted between these two hypersurfaces are of the same length. In particular, hypersurfaces in $\mathcal F$ are a family of parallel ones and can be parametrized or labeled by the arc-length parameter.
\end{prop}

\begin{rmk}\rm Clearly, by Proposition \ref{prop4.1}, the arc-length function $s$ of the geodesics $\gamma_\xi$ well defines a one-form $ds$ globally on $N_*$.

Each hypersurface of the foliation $\mathcal F$, with $\xi$ being its unit normal, was originally named to be a isoparametric hypersurface of $(N,\td g)$.
\end{rmk}

Analysis of the equiaffine parallel hypersurfaces motivates us to introduce the concept of equiaffine isoparametric hypersurfaces as follow:

\begin{dfn}\label{dfnisopfun}
Let $U\subset A^{n+1}$ be a non-empty open domain.
A smooth function $F:U\to\bbr$ is called {\em equiaffine isoparametric}, if the following conditions are satisfied:

(1) There exists a Riemannian metric $\td g$ on $U$ admissible to the volume measure $\vol$, such that $F$ is an (Riemannian) isoparametric function on the Riemannian manifold $(U,\td g)$;

(2) The Hessian $\hess_0 F$ of $F$ with respect to the standard flat connection is semi-negative definite with $\bbr \xi$ being the null space, where $\xi=\fr1a\td\nabla F$;

(3) $\left(\td g+\fr1a\hess_0F\right)|_{\xi^\bot\times\xi^\bot}\equiv 0$.

(4) $2\td\nabla\xi=d\xi-(d\log a)\xi$.
\end{dfn}

\begin{rmk}\rm It is not hard to see that, by using the well-defined one-form $ds$, we can replace conditions (2) and (3) in Definition \ref{dfnisopfun} by a new equation
\be\label{2and3}
\td g+\fr1a\hess_0F=ds^2.
\ee
\end{rmk}

\begin{rmk}\label{rmk4.2}\rm From the definition one easily sees that, if $F$ is an equiaffine isoparametric function then, for any constants $c_1,c_2\in\bbr$ with $c_1>0$, $\td F=c_1F+c_2$ is also an equiaffine isoparametric function. Moreover, $\td F$ and $F$ have the same regular level hypersurfaces. In particular, $\td\xi=\xi$.
\end{rmk}

Now we are in a position to give a proof for our main theorem (Theorem \ref{main1}) in this paper.

{\it Proof of Theorem \ref{main1}}

We first prove the necessity part of the theorem. 

Suppose that $x:M^n\to A^{n+1}$ is an affine isoparametric hypersurfaces with the affine normal vector $Y$, the Blaschke metric $G$ and the affine second fundamental form $B$. Then, by Definition, $x$ is locally strongly convex, or the same, the metric $G$ is positive definite. Furthermore, from Definition \ref{isop hyps} and Corollary \ref{cor3.1}, it follows that there must be a $\delta>0$ such that, for each $\mu\in (-\delta,\delta)$, the hypersurface
$\tlmux=x+\mu Y$ is equiaffine parallel to $x$ with constant affine mean curvature $\tlmul_1$.

Define
$$
U=\{\tlmux(u);\ u\in M^n,\ \mu\in(-\delta,\delta)\}.
$$
Then $U$ is an open neighbourhood of $x(M^n)$, and a Riemannian metric $\td g$ on $U$ can be given as follows:

Let $(u^1,\cdots,u^n)$ be a local coordinate system. Then $(u^1,\cdots,u^n,u^{n+1}:=\mu)$ is a local coordinate system for $U$. In particular, $\pp{}{\mu}=Y$. By Proposition \ref{prop3.1}, the affine normal vector $\tlmuy=cY$ with $c\equiv c(\mu)$, a smooth function of the parameter $\mu\in(-\delta,\delta)$. Using the Blaschke metric $\tlmug=\sum \tlmug_{ij}\omega^i\omega^j$ of $\tlmux$ and the function $c(\mu)$, a Riemannian metric $\td g$ can be introduced on $U$ by
$$
\td g=\sum \tlmug_{ij}\omega^i\omega^j+\fr1{c^2}d\mu^2.
$$
Equivalently speaking, $\td g$ is the Riemannian metric on $U$ which corresponds to a pseudo-product of $\tlmug$ and $d\mu^2$ on $M^n\times (-\delta,\delta)$. For simplicity, we shall always use the index $\mu$ to denote the index $n+1$, say,
$$\td g_{\mu\mu}\equiv\td g_{n+1\,n+1}=\td g(Y,Y)=\fr1{c^2}.$$
From the discussion of Section \ref{s3}, we know that for all $i,j$,
$$\tlmug_{ij}=\fr1c(G_{ij}-\mu B_{ij}),\quad\text{where}\quad G_{ij}=G\left(\pp{}{u^i},\pp{}{u^j}\right),\quad B_{ij}=B\left(\pp{}{u^i},\pp{}{u^j}\right).$$
Denote by $\tlmutdgm^k_{ij}$ the coefficients of the Levi-Civita connection of $\tlmug$. Then we easily find the coefficients $\td\Gamma^C_{AB}$, $\forall\, A,B,C$, of Levi-Civita connection of $\td g$ as follows:
\begin{align}
\td\Gamma^k_{ij}=&\tlmutdgm^k_{ij},\quad \td\Gamma^\mu_{ij}=\fr12cB_{ij},\quad \forall\,i,j,k,\label{tdgamm1}\\ \td\Gamma^i_{\mu j}=&\td\Gamma^i_{j\mu }=-\fr1{2c}\sum \tlmug^{ik}B_{kj}=-\fr1{2c}\,\tlmub^i_j,\quad \forall\,i,j,\label{tdgamm2}\\
\td\Gamma^i_{\mu\mu}=&\td\Gamma^\mu _{i\mu}=\td\Gamma^\mu_{\mu i}=0,\quad\forall\,i,\label{tdgamm3}\\
\td\Gamma^\mu_{\mu\mu}=&-\pp{}{\mu}\log c.\label{tdgamm4}
\end{align}

Define a smooth function $F:U\to\bbr$ by $F(\tlmux(u))=\mu$. Then $F_i=0$, $F_\mu=1$. By using this and \eqref{tdgamm1}$\sim$\eqref{tdgamm4}, we easily find that, with respect to the metric $\td g$ on $U$, the gradient and the Laplacian of $F$ are as follows:
\begin{align}
\td\nabla^{\td g}F=&c\,\tlmuy,\\
\td\Delta_{\td g}F=&\sum \td g^{AB}\left(F_{AB}-F_C\td\Gamma^C_{AB}\right) =-\sum\tlmug^{ij}\td\Gamma^\mu_{ij}-\td g^{\mu\mu}\td\Gamma^\mu_{\mu\mu}\nnm\\ =&-\fr12c\sum\tlmug^{ij}B_{ij}+cc'\equiv-\fr12nc\,\tlmul_1 +cc'.
\end{align}
Since $\tlmul_1$ is constant along $\tlmux(M^n)$, it depends only on the parameter $\mu$. Therefore, $\td\Delta_{\td g}F$ is a function of $F$. So, by the fact that
$$|\td\nabla^{\td g}F|^2_{\td g}=\td g(c\,\tlmuy,c\,\tlmuy)=c^2,$$
$F$ is an isoparametric function on the Riemannian manifold $(U,\td g)$ with
$$a(t)\equiv c(t),\quad b(t)\equiv -\fr12nc(t)\,\tlmul_1+c(t)c'(t).$$
In particular,
$$\xi=\fr1a\nabla^{\td g}F=\tlmuy,\quad s=\int_0^\mu \fr{dt}{c(t)}.$$

Furthermore, the Hessian $\hess_0F(\tlmux(u))$ of $F$ with respect to the standard flat connection $\mathring{\Gamma}^C_{AB}$ can be computed as
\begin{align}
(\hess_0F)_{ij}(\tlmux(u))=&F_{ij} -F_\mu\mathring{\Gamma}^\mu_{ij} =-c\,\tlmug_{ij}(u)=-c\td g_{ij}(u,\mu),\quad\forall\,i,j;\\
(\hess_0F)_{i\mu}(\tlmux(u))=&(\hess_0F)_{\mu i}(\tlmux(u))=(\hess_0F)_{\mu\mu}(\tlmux(u))=0,\quad\forall\,i.
\end{align}
So \eqref{2and3} holds since $a=c$.

On the other hand, since $\fr1c\xi=\fr1c\tlmuy=Y$ is a constant vector along $\mu$-curves which are both geodesics in $(U,\td g)$, with the arc-length $s$, and straight lines in $A^{n+1}$, we find by \eqref{tdgamm1}$\sim$\eqref{tdgamm4} that
\begin{align*}
2\td\nabla_i\xi=&2c\td\nabla_iY=c\sum\td\Gamma^j_{\mu i}\tlmux_j=-\tlmub^j_i\,\tlmux_j=\tlmuy_i =\pp{}{u^i}\xi-\pp{}{u^i}(\log c)\xi,\quad\forall\,i,\\
2\td\nabla_\mu\xi=&2\sum\td\Gamma^i_{\mu\mu}\tlmux_i=0,\quad \pp{}{\mu}\xi-\pp{}{\mu}(\log c)\xi=c'Y-\fr1cc'cY=0.
\end{align*}
So, the condition (4) in Definition \ref{dfnisopfun} is also met by $F$.

Now we have proved that $F$ is an equiaffine isoparametric function. Moreover, it is clear that each $\tlmux$ including $x\equiv {}^0x$ is a regular set of $F$.

Next we prove the sufficiency part of Theorem \ref{main1}.

Suppose that $x:M^n\to A^{n+1}$ is a regular level set of an equiaffine isoparametric function $F$ defined on an open domain $U\subset A^{n+1}$ where, by Definition \ref{dfnisopfun}, $F$ is first all an isoparametric function on the Riemannian manifold $(U,\td g)$ with some Riemannian metric $\td g$. From Proposition \ref{prop4.1}, we know that $x$ is one of the parallel hypersurfaces of constant mean curvature in the foliation $\mathcal F$, for which the unit normal $\xi:=\fr1a\td\nabla F$ is tangent to and parallel along each of its integral curves $\gamma_\xi$, that is, these curves $\gamma_\xi$ all unit-speed geodesics  on $(U,\td g)$ with respect the arc-length parameter $s$. So by Condition (4) in Definition \ref{dfnisopfun}, $Y:=\fr1a\xi$ is constant as $\bbr^{n+1}$-valued function along $\gamma_\xi$. Thus, all the curves $\gamma_\xi$ are straight lines in $A^{n+1}$ and can be parametrized, starting from points on $x(M^n)$, as $\gamma_{\xi(u)}(\mu)=x(u)+\mu Y$, $\mu\in(-\delta,\delta)$ with some $\delta>0$, for each $u\in M^n$.

Consequently, starting from the given hypersurface $x$, the above family of parallel hypersurfaces in $\mathcal F$ can be parametrized by the parameter $\mu$ simply as
\be\label{mux}\tlmux(u)=x(u)+\mu\xi(u),\quad u\in M^n,\ee
for $\mu\in (-\delta,\delta)$.

We shall prove that, for each $\mu\in (-\delta,\delta)$, $\tlmux$ is a locally strictly convex hypersurface with constant affine mean curvature $\tlmul_1$. For doing this, we take the following three steps:

(1) The induced metric on $M^n$ by the immersion $\tlmux:M^n\to (U,\td g)$ is equal to the affine metric $\tlmug$.

In fact, for any tangent frame field $\{e_1,\cdots,e_n\}$ on $M$, let $\lambda=(\det(\tlmux^*\td g_{ij}))^{\fr12(n+2)}$ and $e_{n+1}=\lambda^{-\fr1{n+2}}\xi$. Then by the Lagrange identity and the fact that $\td g$ is admissible to the measure $\vol$, we find
\begin{align*}
&\vol(e_1,\cdots,e_n,e_{n+1})=\vol_{\td g}(e_1,\cdots,e_n,e_{n+1})\\
=&(\det(\td g_{AB}))^{\fr12}
=(\det(\tlmux^*\td g_{ij}))^{\fr12}|e_{n+1}|_{\td g} =(\det(\tlmux^*\td g_{ij}))^{\fr12}\lambda^{-\fr1{n+2}}|\xi|_{\td g}
=1.
\end{align*}
Thus $\{e_1,\cdots,e_n,e_{n+1}\}$ is an affine Darbourx frame field. Moreover, it holds by the definition of the affine metric that
\be\label{sxij}
\tlmux_{ij}=\sum\tlmugm^k_{ij}\tlmux_k+\tlmu{h}_{ij}e_{n+1}
=\sum\tlmugm^k_{ij}\tlmux_k+\tlmug_{ij}\tlmuh^{\fr1{n+2}} \lambda^{-\fr1{n+2}}\xi,\quad\forall\,i,
\ee
with $\tlmuh=\det(\tlmu{h}_{ij})$.

On the other hand, since the hypersurface $\tlmux(M^n)$ is a regular level set of the function $F$, we have that $F(\tlmux^1,\cdots,\tlmux^{n+1})=\const$. Take differentiation of this equation twice, we find that
$$
\sum\ppp{F}{x^A}{x^B}\tlmux^A_i\tlmux^B_j+\sum \pp{F}{x^A}\tlmux^A_{ij}\equiv 0,\quad\forall\,i,j
$$
or
\be\label{hessfij}
\tlmux^*(\hess_0F)_{ij}+\sum F_A\tlmux^A_{ij}\equiv 0,\quad\forall\,i,j
\ee
for each $\mu\in (-\delta,\delta)$.

Now, by using \eqref{sxij}, \eqref{hessfij} and \eqref{2and3}, we obtain that, for all $i,j$,
\begin{align}
\tlmug_{ij}\tlmuh^{\fr1{n+2}}\lambda^{-\fr1{n+2}}=&\td g(\tlmux_{ij},\xi)=\td g(\tlmux_{ij},\fr1a\nabla^{\td g}F)
=\fr1a\sum\td g_{AB} \tlmux^A_{ij}(\nabla^{\td g}F)^B\nnm\\
=\fr1a\sum \tlmux^A_{ij}F_A
=&-\tlmux^*\left(\fr1a\hess_0F\right)_{ij}=x^*(\td g-ds^2)_{ij} =(x^*\td g)_{ij}.\label{tdg=afg}
\end{align}
The equality \eqref{tdg=afg} means that $\tlmug$ is positive definite, that is, $\tlmux$ is locally strongly convex for each $\mu\in (-\delta,\delta)$. Moreover, by the definition of the affine metric $\tlmug$, we have $\det (\tlmug_{ij})=\tlmuh^{\fr2{n+2}}$ (see \eqref{detg}). Using this we take the determinant of \eqref{tdg=afg} and obtain
$$
\tlmuh\lambda^{-\fr n{n+2}}=\det(\tlmug_{ij})\tlmuh^{\fr n{n+2}}\lambda^{-\fr n{n+2}}=\det(x^*(\td g)_{ij})=\lambda^{\fr2{n+2}},
$$
implying that $\lambda=\tlmuh$. Putting this into \eqref{tdg=afg} we have finally proved that $\tlmux^*\td g=\tlmug$.

(2) The affine normal vector $\tlmuy$ of $\tlmux$ at any $u\in M^n$ is equal to $\xi(\tlmux(u))$ for each $\mu\in (-\delta,\delta)$.

In fact, we can choose an affine Darbourx frame $\{e_1,\cdots,e_n,e_{n+1}\}$ such that $e_{n+1}\equiv \xi$. In this case, by the condition (4) in Definition \ref{dfnisopfun}, we have along $\tlmux(M^n)$ that
\be
de_{n+1}=d\xi=2\td\nabla\xi=-2\sum \tlmua^j_i\omega^i\,\tlmux_j\label{den+1},
\ee
where $\tlmua^j_i$ are the components of the Weingarten map of $\tlmux$. This shows that $\omega^{n+1}_{n+1}\equiv 0$. On the other hand, since $\td g(\xi,\xi)\equiv 1$, we can make a choice of $\{e_1,\cdots,e_n\}$ such that $\tlmuh=\det(\tlmu{h}_{ij})=\const$, say, by choosing $\{e_1,\cdots,e_n\}$ to be orthonormal with respect to the affine metric $\tlmug$ which, as shown in (1), is now exactly the induced metric of $\td g$. Consequently, we have
$$\omega^{n+1}_{n+1}+\fr1{n+2}d\log H\equiv 0,$$
which implies, by Proposition \ref{prop2.1}, that $e_{n+1}$ or $\xi$ is parallel to the affine normal vector $\tlmuy$, and
\be\label{trdiftsr}
\sum\tlmug^{ij}\,\tlmua^k_{ij}\equiv \sum\tlmug^{ij}\left(\tlmugm^k_{ij}-\tlmutdgm^k_{ij}\right)=0.
\ee
In particular, if $\{e_1,\cdots,e_n\}$ is chosen to be orthonormal w.r.t. $\tlmug$, then we have $$1=\det(\tlmug_{ij})=\tlmuh^{\fr2{n+1}},$$
that is, $\tlmuh\equiv 1$ and $\tlmu{h}_{ij}=\tlmug_{ij}=\delta_{ij}$. Using \eqref{sxij} and \eqref{trdiftsr} we find
\begin{align*}
\tlmuy=&\fr1n\Delta_{\tlmug}\tlmux =\fr1n\sum\tlmug^{ij}\left(\tlmux_{ij} -\sum\tlmux_k\tlmutdgm^k_{ij}\right)\\ =&\fr1n\left(\sum\tlmug^{ij}\left(\tlmugm^k_{ij} -\tlmutdgm^k_{ij}\right)\tlmux_k +\sum\tlmug^{ij}\,\tlmug_{ij}e_{n+1}\right)\\
=&e_{n+1}=\xi.
\end{align*}

By Remark \ref{rmk4.2}, we can suitably choose the constant $c_1$, if necessary, such that $a(0)=1$. It then follows that, $Y=\fr1{a(0)}{}^0Y={}^0Y$ is the affine normal vector of the original hypersurface $x\equiv{}^0x$.
So we can conclude that
the hypersurface $x\equiv {}^0x$ is among a family \eqref{mux} of equiaffine parallel hypersurfaces in $A^{n+1}$.

(3) All hypersurfaces $\tlmux$ in \eqref{mux} are of constant affine mean curvature.

In fact, we use once again \eqref{den+1} and compare it with \eqref{yi} to obtain that
$$
-\sum\tlmub^j_i\tlmux_j=\tlmuy_i=\xi_i=2\td\nabla_{e_i}\xi =-2\sum\tlmua^j_i\tlmux_j,\quad\forall\,i,
$$
which proves that $\tlmua^j_i=\fr12\tlmub^j_i$ for all $i,j$ and $\mu$. Take the trace we find that, for each $\mu$, the mean curvature of the isometric immersion $\tlmux:M^n\to (U,\td g)$ is equal to one half of the affine mean curvature $\tlmul_1$ of the locally strongly convex hypersurface $\tlmux:M^n\to A^{n+1}$. But, since the former $\tlmux$ is the regular level set of the (Riemannian) isoparametric function $F$ on $(U,\td g)$, being of constant mean curvature, the latter $\tlmux$ must be of constant affine mean curvature, for each $\mu\in(-\delta,\delta)$.

Summing up conclusions (1), (2) and (3), the sufficiency part of the theorem is proved.

Now we give a proof of Theorem \ref{main2}.

{\it Proof of Theorem \ref{main2}}

It suffices to show that all the affine principal curvatures $\lambda_1,\cdots,\lambda_n$ of $x$ are constant if each of its equiaffine parallel hypersurfaces is of constant affine mean curvature. To do this, we use \eqref{mulmbd} to find
\be\label{mul1}
\tlmul_1=\fr1n\sum\fr{c\lambda_i}{1-\mu\lambda_i},\quad \text{for\ \ }\mu\in (-\delta,\delta),
\ee
that is,
\be\label{mul1c}
\sum\fr{\lambda_i}{1-\mu\lambda_i}=\fr ncL_1(\mu),\quad \mu\in (-\delta,\delta).
\ee
If $\tlmul_1$ is constant for all $\mu\in (-\delta,\delta)$, then right hand side of \eqref{mul1c} is a smooth function of $\mu$. Taking the differentiation of \eqref{mul1c} by $(n-1)$-times gives
\be\label{dmul1}
\sum\fr{\lambda^\iota_i}{(1-\mu\lambda_i)^\iota} =\left(\fr1cL_1(\mu)\right){}^{(\iota)},\quad \mu\in (-\delta,\delta),\quad \iota=2,\cdots,n.
\ee
Evaluating \eqref{mul1c} and \eqref{dmul1} at $\mu=0$ simply gives
$$
\sum \lambda^\iota_i=\const,\quad\iota=1,2,\cdots,n,
$$
which implies that $\lambda_i=\const$ for each $i$.

This completes the proof of Theorem \ref{main2}.

The following two corollaries are direct:

\begin{cor}\label{cor4.1}
A non-degenerate hypersurface in the affine space $A^{n+1}$ is equiaffine isoparametric if and only if it is among a family of equiaffine parallel hypersurfaces that are of constant affine principal curvature.
\end{cor}

In fact, this is a direct consequence of Proposition \ref{prop3.5}.

\begin{cor}\label{cor4.2}
A strongly convex hypersurface in the affine space $A^{n+1}$ is an (equi)affine hypersphere if and only if it is among a family of equiaffine parallel affine hyperspheres of the same type.
\end{cor}

In fact, by \eqref{mulmbd}, all the affine principal curvatures $\tlmulmbd_i$ are equal to each other if and only if those $\lambda_i$ of $x$ are.

\begin{rmk}\rm To end this paper, we would like to give some final remarks: Firstly,
Theorem \ref{main2} can be equivalently and directly obtained by \eqref{dett=const}. In fact, from \eqref{dett=const} one easily sees that the two conditions for the equiaffine parallel hypersurfaces are rather restrictive. Secondly, it turns out that, in the definition of equiaffine isoparametric hypersurfaces (Definition \ref{isop hyps}), the adjective subordinate claus ``that are of constant affine mean curvature'' can be really deleted! For example, Corollary \ref{cor4.1} can be simply restated as

\begin{cor}\label{cor4.1'}
A non-degenerate hypersurface in the affine space $A^{n+1}$ is equiaffine isoparametric if and only if it is among a family of equiaffine parallel hypersurfaces.
\end{cor}

\end{rmk}

\end{document}